\documentclass[11pt]{amsart}

\usepackage{amsmath, amsthm, amssymb}
\usepackage{amsfonts}
\usepackage{mathtools}

\usepackage[cal=boondox]{mathalfa}
\usepackage{bbm}
\usepackage{mathrsfs}
\usepackage{stmaryrd}
\usepackage{esint}
\usepackage[T1]{fontenc}
\usepackage{lmodern}

\usepackage{graphicx}
\usepackage{tikz}
\usepackage{tikz-cd}
\usetikzlibrary{arrows}
\usepackage[all,cmtip]{xy}
\usepackage{amscd}
\usepackage{picinpar}

\usepackage{enumerate}
\usepackage{enumitem}
\usepackage{comment}
\usepackage{array,booktabs,longtable,ragged2e}
\usepackage{extarrows}
\usepackage{verbatim}
\usepackage{calc}
\usepackage{xcolor}

\usepackage{hyperref}

\numberwithin{equation}{section}

\theoremstyle{plain}
\newtheorem{thm}{Theorem}
\numberwithin{thm}{subsection}

\newtheorem{lem}[thm]{Lemma}
\newtheorem{prop}[thm]{Proposition}

\newtheorem{mainthm}{Theorem}

\theoremstyle{definition}

\theoremstyle{remark}
\newtheorem{rem}[thm]{Remark}

\newcolumntype{L}[1]{>{\RaggedRight\arraybackslash}p{#1}}
\DeclareMathOperator{\GL}{GL}

\DeclareMathOperator{\SO}{SO}

\DeclareMathOperator{\PU}{PU}
\DeclareMathOperator{\PGL}{PGL}
\DeclareMathOperator{\PSL}{PSL}

\DeclareMathOperator{\Bir}{Bir}

\DeclareMathOperator{\Aut}{Aut}

\DeclareMathOperator{\coh}{H}

\DeclareMathOperator{\Isom}{Isom}

\DeclareMathOperator{\Jonq}{Jonq}
\DeclareMathOperator{\Steinberg}{St}
\DeclareMathOperator{\Supp}{Supp}

\def\bC{{\mathbb{C}}}

\def\bR{{\mathbb{R}}}

\def\bF{{\mathbb{F}}}
\def\bH{{\mathbb{H}}}

\def\bZ{{\mathbb{Z}}}

\def\bP{{\mathbb{P}}}
\def\bN{{\mathbb{N}}}
\def\bQ{{\mathbb{Q}}}

\def\sC{{\mathscr{C}}}

\def\sF{{\mathscr{F}}}
\def\sG{{\mathscr{G}}}
\def\sH{{\mathscr{H}}}

\def\sK{{\mathscr{K}}}
\def\sL{{\mathscr{L}}}

\def\sO{{\mathscr{O}}}
\def\sP{{\mathscr{P}}}

\def\sU{{\mathscr{U}}}

\def\sW{{\mathscr{W}}}

\def\fJ{{\mathfrak{J}}}
\def\fT{{\mathfrak{T}}}

\def\aE{{\mathcal{E}}}
\def\aO{{\mathcal{O}}}

\newcommand{\Induce}[2]{\operatorname{Ind}_{#1}^{#2}}
\newcommand{\Restrict}[2]{\operatorname{Res}_{#1}^{#2}}

\begin{document}

\title[Cremona rigidity]{On cube and Cremona rigidity for higher-rank lattices}

\author{Shengyuan Zhao}
\address{Universit\'e Paul Sabatier, Institut de Math\'ematiques de Toulouse, 118, route de Narbonne, F-31062 Toulouse, France}
\email{shengyuan.zhao@math.univ-toulouse.fr}
\dedicatory{In memory of Misha Kapovich}

\begin{abstract}
For irreducible lattices in semisimple Lie groups of real rank at least $2$, we prove a cohomological vanishing result implying that any action on a CAT(0) cube complex fixes a vertex whenever every hyperplane stabilizer is solvable. As an application, we prove regularizability for all actions of higher-rank lattices by birational transformations on projective surfaces. We first use superrigidity for actions on infinite-dimensional real hyperbolic spaces to reduce to the de Jonquières group, and then apply our fixed-point theorem to the Jonquières complex. Our proof bypasses the direct use of property FW. 
\end{abstract}

\maketitle

\section{Introduction}

\subsection{Zimmer's Program}
Motivated by rigidity results for irreducible lattices in semisimple Lie
groups of higher rank, especially Margulis superrigidity and Zimmer cocycle superrigidity, Zimmer~\cite{ZimmerProgramICM} formulated conjectures predicting that actions of higher-rank lattices on manifolds are highly constrained. Rigidity problems for non-linear actions of lattices are now known as Zimmer's program. We refer to the recent works \cite{BFH1}, \cite{BFH2}, \cite{CantatXie}, \cite{DSVX}, and to the references therein, for historical accounts and various formulations.

This paper concerns Zimmer's program in algebraic geometry, where one studies actions on projective varieties by automorphisms, or more generally by birational transformations. Over the complex numbers, the automorphism case for smooth projective varieties is a special case of the original Zimmer program for diffeomorphisms of compact manifolds; see \cite{CantatZimmer,CantatZeghib}. The birational case is more intriguing. On the one hand indeterminacy loci make the action on the variety non-set-theoretic and may cause non-compact phenomena. On the other hand, the algebraicity can impose stronger rigidity than in the differentiable setting.

Progress on the birational Zimmer program has been made in \cite{Deserti,CantatTits,CantatXie,CantatCornulier,CornulierCremona,LonjouUrech}. In dimension two, Cantat~\cite[\S~4]{CantatTits} classified all birational actions of lattices with Kazhdan's property~(T). In arbitrary dimension, Cantat--Xie~\cite{CantatXie} proved that higher-rank lattices with Kazhdan's property~(T) and the congruence subgroup property admit no infinite birational action on varieties whose dimension is too small compared with the real rank of the lattice. These two papers are closest in spirit to Zimmer's original differentiable conjectures. By contrast, \cite{CantatCornulier,CornulierCremona,LonjouUrech} deal with lattices having property FW, which is strictly weaker than Kazhdan's property~(T), and their results have a somewhat different flavor, discussed later. 

Most of the literature on Zimmer's program focuses on lattices in simple Lie groups of real rank at least \(2\), or more generally on lattices in semisimple Lie groups whose non-compact simple factors all have real rank at least \(2\).
Such lattices have Kazhdan's property~(T); see \cite{BHV}. Lattices in
semisimple Lie groups with rank-one factors often present additional
difficulties. Nevertheless, the general philosophy that higher-rank lattices should have constrained dynamics does not exclude them; see, for instance, \cite{GhysCircle} for a case of Zimmer-type rigidity in the presence of rank-one factors. Cornulier~\cite{CornulierLattices} proved that many lattices without Kazhdan's property~(T) still have property FW, and conjectured this for all higher-rank lattices. Thus \cite{CantatCornulier,CornulierCremona,LonjouUrech} suggest that many birational rigidity results should extend from lattices with property~(T) to all higher-rank lattices. We confirm this in dimension two, including cases where property FW is not presently known.

\subsection{Regularizability}
We refer to \cite{BeauvilleBook} for background on birational geometry of surfaces, and to \cite{BlancAlgGroup,CantatMilnor,CantatTits} for the
background results on groups of birational transformations of surfaces. For simplicity, we work over the field of complex numbers.

Let $G< \Bir(Y)$ be a group of birational transformations on a projective variety $Y$. We say that $G$ is \emph{regularizable} if there exists a quasi-projective variety $Z$ and a birational map $f:Y\dashrightarrow Z$ such that $fGf^{-1}$ is a group of automorphisms of $Z$. If $Z$ can moreover be chosen projective, then $G$ is \emph{projectively regularizable}. If, in addition, $Z$ can be chosen smooth projective and there exists a finite-index subgroup $G'<G$ such that $fG'f^{-1}$ is contained in $\Aut^0(Z)$, the connected component of identity of $\Aut(Z)$, then $G$ is \emph{virtually isotopically projectively regularizable}. 

If $Z$ is a smooth complex projective variety, then $\Aut^0(Z)$ is a
connected Lie group. Since homomorphisms from higher-rank lattices to
Lie groups are strongly constrained by Margulis superrigidity, the
classification of birational actions of higher-rank lattices naturally splits into two parts. The first is to prove that every infinite birational action is virtually isotopically projectively regularizable. The second is to determine which connected algebraic groups can occur as $\Aut^0(Z)$. The second problem is a classification problem of algebraic varieties and is completely understood in dimension two; see \cite{BlancAlgGroup}. We shall therefore focus on the regularization problem.

Cantat--Cornulier~\cite{CantatCornulier} showed that regularization problems in birational geometry are closely related to property FW. Subsequently, Cornulier~\cite{CornulierCremona} and Lonjou--Urech~\cite{LonjouUrech} proved that every birational action of a group with property FW is regularizable, and is projectively regularizable in dimension two. Cornulier~\cite{CornulierLattices,CornulierFW} also proved that many lattices without Kazhdan's property~(T) have property FW; this includes irreducible lattices in semisimple Lie groups having at least one factor with property~(T). Thus, in dimension two, remaining examples include, for instance, uniform irreducible lattices in $\PSL(2,\bR)^n\times \PGL(2,\bC)^m$ and irreducible lattices in $\PU(1,n)^m$.

\subsection{Main results}\label{subsec:main-results}
Let $G_1,\cdots,G_n$ be non-compact connected center-free simple linear algebraic groups. Let $G=G_1\times \cdots \times G_n$ and assume that the real rank of $G$ is $\ge 2$. Let $\Gamma < G$ be an irreducible lattice.

\begin{mainthm}\label{mainthm:Lattice-regularizable}
Let $S$ be a complex projective surface and let $\rho:\Gamma\to \Bir(S)$ be a homomorphism. Then $\rho(\Gamma)$ is virtually isotopically projectively regularizable. 
\end{mainthm}
As explained above, Theorem~\ref{mainthm:Lattice-regularizable}, together
with Margulis superrigidity and the classification~\cite{BlancAlgGroup}
of actions of connected algebraic groups on projective surfaces, gives in
principle a list of all possible birational actions of higher-rank lattices on projective surfaces, up to birational conjugacy. We shall not give this list explicitly. The only case requiring (arguably) a separate discussion concerns actions of lattices in $\PSL(2,\bR)^n\times \PGL(2,\bC)^m$ on Hirzebruch surfaces $\sF_k, k>1$; this case is described briefly in \S\ref{sec:hirzebruch}. The following is a short statement in the spirit of Zimmer's original conjectures: 
\begin{mainthm}\label{mainthm:Zimmer-no-action}
Let $S$ be a complex projective surface. Assume that none of $G_1,\cdots,G_n$ is $\PSL(2,\bR)$, $\PGL(2,\bC)$, $\PU(1,2)$, $\PSL(3,\bR)$ or $\PGL(3,\bC)$. Then there is no homomorphism with infinite image from $\Gamma$ to $\Bir(S)$.
\end{mainthm}

The proof of Theorem~\ref{mainthm:Lattice-regularizable} uses the following two results of independent interest, proved using Borel--Serre duality and Steinberg modules.
\begin{mainthm}\label{mainthm:cohomology-vanishing}
There exists a finite-index torsion-free subgroup $\Gamma'<\Gamma$ such that, for any solvable subgroup $K\subset \Gamma'$, for any field $\bF$,  
\[
\coh^1(\Gamma',\bF[\Gamma'/K])=0.
\]
\end{mainthm}

\begin{mainthm}\label{mainthm:cube-complex}
If $\Gamma$ acts on a CAT(0) cube complex $\sC$ and if the stabilizer of every hyperplane of $\sC$ is solvable, then $\sC$ has a $\Gamma$-fixed vertex. 
\end{mainthm}

Cornulier~\cite[Prop.~7.I.3]{CornulierFW} showed that property FW for $\Gamma$ is equivalent to the vanishing $\coh^1(\Gamma,\bZ[\Gamma/K])=0$ for every subgroup $K$. In particular, for lattices whose property FW is known, Theorem~\ref{mainthm:cohomology-vanishing} and~\ref{mainthm:cube-complex} are merely special consequences of Cornulier's theorems~\cite{CornulierLattices}. However our proof is very different.
The weaker vanishing statement in Theorem~\ref{mainthm:cohomology-vanishing} turns out to be sufficient for the proof of Theorem~\ref{mainthm:Lattice-regularizable}, thanks to the following result which is a consequence of Monod's superrigidity~\cite{MonodRigidity} and a result of Caprace--Lytchak~\cite{CapraceLytchak}:
\begin{mainthm}\label{mainthm:hyperbolic-lattice-rigidity}
Assume that $\rho:\Gamma\to \Isom(\bH^\aE)$ is a non-elementary isometric action on an infinite-dimensional real hyperbolic space $\bH^\aE$. There exists a unique minimal $\rho$-invariant closed hyperbolic subspace $\bH^\rho\subset \bH^\aE$. There exist $i\in \{1,\cdots,n\}$ and a continuous homomorphism $\rho_i:G_i\to \Isom(\bH^\rho)$ such that the restriction of $\rho$ to $\bH^\rho$ factors through the composition of $\rho_i$ with the projection $G\to G_i$.
\end{mainthm}

\section{Cohomology vanishing via duality}\label{sec:duality}

\subsection{Steinberg module}\label{subsec:Steinberg}
Let $\sG$ be a reductive algebraic group defined over $\bQ$ of semisimple $\bQ$-rank $q$. The Tits building $\fT(\sG)$ is a $(q-1)$-dimensional simplicial complex whose $k$-dimensional simplices are in bijection with parabolic subgroups of semisimple rank $(q-1-k)$. The group $\sG(\bQ)$ acts on $\fT(\sG)$ by conjugation on itself, which permutes the parabolic subgroups. See \cite{BrownBuildings}. 

Let $\bF$ be a coefficient field and consider the reduced homology group $\widetilde{\coh}_{q-1}(\fT(\sG),\bF)$. By definition the \emph{Steinberg module} $\Steinberg (\sG,\bF)$ is $\widetilde{\coh}_{q-1}(\fT(\sG),\bF)$ equipped with its natural $\bZ \sG(\bQ)$-module structure. 

By Solomon--Tits theorem (see \cite[Th.~IV.5.2]{BrownBuildings}), the complex $\fT(\sG)$ has the homotopy type of a bouquet of $(q-1)$-spheres. Let $\Omega(\sG)$ be the set of minimal parabolic subgroups of $\sG$. Then the chain group satisfies $\operatorname{C}_{q-1}(\fT(\sG),\bF)\cong \bF[\Omega(\sG)]$, and we have an exact sequence of $\bZ \sG(\bQ)$-modules
\begin{align}
& 0  \to
\Steinberg({\sG,\bF})
\to
\operatorname{C}_{q-1}(\fT(\sG), \bF) \cong \bF[\Omega(\sG)]
\to
\operatorname{C}_{q-2}(\fT(\sG),\bF)
\to
\cdots \notag
\\
& \to
\operatorname{C}_0(\fT(\sG),\bF)
\to
\bF
\to 0.\label{eq:Solomon-Tits}
\end{align}

We will also need the following (see also \cite[Th.~2.1]{APS}):
\begin{thm}[Reeder~{\cite[Prop.~1.1]{Reeder}}]\label{thm:Reeder}
Let $\sP$ be a parabolic subgroup and $\sL$ be a Levi factor of $\sP$. There exists an isomorphism of $\bZ\sP(\bQ)$-modules 
\[
\Induce{\sL(\bQ)}{\sP(\bQ)}\Steinberg(\sL,\bF)\cong 
\Restrict{\sP(\bQ)}{\sG(\bQ)}\Steinberg(\sG,\bF).
\]
\end{thm}
The isomorphism in Theorem~\ref{thm:Reeder} is not canonical; we refer to the one Reeder constructed whenever we use the theorem.

\subsection{A homology vanishing result}\label{subsec:homology}

For a connected reductive $\bQ$-algebraic group $\sG$, let $\chi(\sG)$ be the group of $\bQ$-characters of $\sG$, $q_{\sG}$ be the semisimple $\bQ$-rank of $\sG$, $r_\sG$ be the real rank of $\sG$, and $c_\sG$ be the $\bZ$-rank of $\chi(\sG)$. Define $\delta(\sG):=r_\sG-c_\sG-q_\sG$.

For our semisimple $\sH$, we have $r_\sH=r,q_\sH=q, c_\sH=0$, so that $\delta(\sH)=r-q$. For every Levi subgroup $\sL$ of a $\bQ$-parabolic subgroup $\sP$ of $\sH$, one has $r_\sL=r$ and $q_\sL+c_\sL=q$, so that $\delta(\sL)=r-q$.

\begin{prop}\label{prop:solvable-steinberg-vanishing-reductive}
Let $\sG$ be a connected reductive group defined over $\bQ$. 
Let $A<\sG(\bQ)$ be a finitely generated torsion-free solvable subgroup
contained in an arithmetic subgroup of $\sG$. Then for every $i>\delta(\sG)$,
\[
\coh_i\bigl(A,\Steinberg(\sG,\bF)\bigr)=0.
\]
\end{prop}

We prove Proposition~\ref{prop:solvable-steinberg-vanishing-reductive} by induction on $q_\sG$, starting with:
\begin{lem}\label{lem:not-contained}
Proposition~\ref{prop:solvable-steinberg-vanishing-reductive} holds under the additional assumption that $A$ is not contained in any proper $\bQ$-parabolic subgroup of $\sG$. 
\end{lem}
\begin{proof}
Assume that $A$ is not contained in any proper $\bQ$-parabolic subgroup of $\sG$. Let $\overline{A}$ be its Zariski closure. Since $A$ is solvable, the identity component $\overline A^{0}$ is solvable. If the unipotent radical of $\overline A^{0}$ were non-trivial, then, by the Borel--Tits theorem~\cite[Prop.~3.1]{BorelTits}, its normalizer would be contained in a proper $\bQ$-parabolic subgroup of $\sG$. Since the unipotent radical is normal in $\overline A^{0}$, this would force $A$ to be contained in a proper rational parabolic subgroup, contrary to our assumption. Hence $\overline A^{0}$ is a torus. 

Since $\overline A^{0}$ is a torus, $A$ is virtually free abelian. Since $A$ is contained in an arithmetic subgroup, all $\bQ$-rational characters of $\sG$ have finite image on $A$. Therefore, $A$ is virtually contained in the intersection of the kernels of all $\bQ$-rational characters. Thus, its cohomological dimension is at most $r_\sG-c_\sG$. So,
\begin{equation}\label{eq:solvable-vanishing-caseone-formula-1}
\coh_j(A,\bF)=0
\qquad\text{for every } j>r_\sG-c_\sG.
\end{equation}

Let $\sigma$ be an $l$-simplex of the Tits building $\fT(\sG)$, and let
$\sP_\sigma$ be its stabilizer, which is a parabolic subgroup. Let $\sL_\sigma$ be a Levi factor of $\sP_\sigma$. Consider
\[
A_\sigma:=A\cap \sP_\sigma(\mathbb Q).
\]
Since $A$ contains no non-trivial unipotent elements, the group $A_\sigma$ injects into $\sL_\sigma$ under the projection $\sP_\sigma\to \sL_\sigma$. Its image is contained in an arithmetic subgroup of $\sL_\sigma$ by \cite[Prop.~4.1]{PlatonovRapinchuk}. Moreover,
\[
r_{\sL_\sigma}=r_\sG,
\qquad
c_{\sL_\sigma}=c_\sG+l+1.
\]
Thus the same rank estimate gives
\begin{equation}\label{eq:solvable-vanishing-caseone-formula-2}
\coh_j(A_\sigma,\bF)=0  \qquad\text{for every }
        j>r_\sG-c_\sG-l-1.
\end{equation}

Now, restricting $\sG(\bQ)$ action to $A$, we regard \eqref{eq:Solomon-Tits} as an exact sequence of $\bF A$-modules:
\begin{align*}
& 0  \to
\Steinberg({\sG,\bF})
\to
\operatorname{C}_{q_\sG-1}(\fT(\sG), \bF)
\to
\operatorname{C}_{q_\sG-2}(\fT(\sG),\bF)
\to
\cdots \notag
\\
& \to
\operatorname{C}_0(\fT(\sG),\bF)
\to
\bF
\to 0.
\end{align*}
For $0\leq l\leq q_\sG-1$, we have an isomorphism of $\bF A$-modules:
\begin{equation}\label{eq:solvable-vanishing-caseone-formula-3}
\operatorname{C}_l(\fT(\sG),\bF)\cong
\bigoplus_{[\sigma]\in A\backslash \fT(\sG)^{(l)}}
 \bF[A/A_\sigma].
\end{equation}
By Shapiro's lemma (see \cite[III.6.2]{BrownBook}) and by \eqref{eq:solvable-vanishing-caseone-formula-2} we get
\begin{equation}\label{eq:solvable-vanishing-caseone-formula-4}
\coh_j\big(A,\operatorname{C}_l(\fT(\sG),\bF)\big)\cong
\bigoplus_{[\sigma]}
\coh_j (A_\sigma,\bF)=0
\quad\text{for } j>r_\sG-c_\sG-l-1. 
\end{equation}

Now define, for $0\leq l\leq q_\sG-1$,
\[
Z_{-1}:=\bF,\quad
Z_l:=\ker\bigl(\operatorname{C}_l(\fT(\sG),\bF)\to \operatorname{C}_{l-1}(\fT(\sG),\bF)\bigr).
\]
Then $Z_{q_\sG-1}=\Steinberg({\sG,\bF})$. There is a short exact sequence
\begin{equation}\label{eq:solvable-vanishing-caseone-formula-5}
0\longrightarrow Z_l \longrightarrow
\operatorname{C}_l(\fT(\sG),\bF)
\longrightarrow Z_{l-1} \longrightarrow 0.
\end{equation}
Fix $i>\delta(\sG)$. For every $0\leq l\leq q_\sG-1$, we have $i+q_\sG-l-1>r_\sG-c_\sG-l-1$, so by \eqref{eq:solvable-vanishing-caseone-formula-4}
\begin{equation}\label{eq:solvable-vanishing-caseone-formula-6}
\coh_{i+q_\sG-l-1} \big(A,\operatorname{C}_l(\fT(\sG),\bF)\big)
=
\coh_{i+q_\sG-l} \big(A,\operatorname{C}_l(\fT(\sG),\bF)\big)=0.
\end{equation}
Therefore the long exact sequence associated with \eqref{eq:solvable-vanishing-caseone-formula-5} gives
\begin{equation}\label{eq:solvable-vanishing-caseone-formula-7}
\coh_{i+q_\sG-l-1} \big(A,Z_l\big)\cong 
\coh_{i+q_\sG-l} \big(A,Z_{l-1}\big).
\end{equation}
Applying \eqref{eq:solvable-vanishing-caseone-formula-7} to $l=q_\sG-1,q_\sG-2,\cdots,0$, we obtain
\begin{equation}\label{eq:solvable-vanishing-caseone-formula-8}
\coh_{i} \big(A,\Steinberg({\sG,\bF})\big)\cong
\coh_{i} \big(A,Z_{q_\sG-1}\big)\cong 
\coh_{i+q_\sG} \big(A,Z_{-1}\big)\cong
\coh_{i+q_\sG} \big(A,\bF\big).
\end{equation}
As $i>\delta(\sG)$, we have $i+q_\sG>r_\sG-c_\sG$, and we deduce from \eqref{eq:solvable-vanishing-caseone-formula-1}
\begin{equation}\label{eq:solvable-vanishing-caseone-formula-9}
\coh_{i} \big(A,\Steinberg({\sG,\bF})\big)\cong
\coh_{i+q_\sG} \big(A,\bF\big)=0.
\end{equation}
This proves the lemma.
\end{proof}

\begin{proof}[Proof of Proposition~\ref{prop:solvable-steinberg-vanishing-reductive}]
We argue by induction on $q_\sG$. If $q_\sG=0$, then $\fT(\sG)$ is empty and $\Steinberg(\sG,\bF)=\bF$. Since $\sG$ has no proper $\bQ$-parabolic subgroup in this case, the result follows from Lemma~\ref{lem:not-contained}. 

Now assume that the proposition is already proved when $q_\sG\le k$ for some $k\in \bN$. Assume that $\sG$ has $q_\sG=k+1$. By Lemma~\ref{lem:not-contained}, we may assume without loss of generality that $A$ is contained in a proper $\bQ$-parabolic subgroup $\sP<\sG$. Let $\sU$ be its unipotent radical and let $\sL$ be a Levi factor, so that $\sP=\sU\rtimes \sL$.

By Theorem~\ref{thm:Reeder} we get an isomorphism of $\bF A$-modules:
\begin{equation}\label{eq:solvable-vanishing-formula-1}
\Restrict{A}{\sG(\bQ)}\Steinberg(\sG,\bF)
\cong
\Restrict{A}{\sP(\bQ)} \Induce{\sL(\bQ)}{\sP(\bQ)}\Steinberg(\sL,\bF).
\end{equation}
By Mackey's formula \cite[III.5.6]{BrownBook}, we obtain
\begin{equation}\label{eq:solvable-vanishing-formula-2}
\Restrict{A}{\sP(\bQ)} \Induce{\sL(\bQ)}{\sP(\bQ)}\Steinberg(\sL,\bF)
\cong
\bigoplus_{[g]\in A\backslash \sP(\bQ)/\sL(\bQ)}
\Induce{A\cap g\sL(\bQ) g^{-1}}{A}\,
g\Steinberg(\sL,\bF).
\end{equation}
We identify $\sP/\sL$ with $\sU$, and denote 
\[
A_u:=A\cap u\sL(\bQ) u^{-1}.
\]
We combine \eqref{eq:solvable-vanishing-formula-1}, \eqref{eq:solvable-vanishing-formula-2} to get an isomorphism of $\bF A$-modules:
\begin{equation}\label{eq:solvable-vanishing-formula-3}
\Restrict{A}{\sG(\bQ)}\Steinberg(\sG,\bF)
\cong
\bigoplus_{[u]\in A\backslash \sP(\bQ)/\sL(\bQ)}
\Induce{A_u}{A}\, u\Steinberg(\sL,\bF).
\end{equation}
Applying Shapiro's lemma to each summand of \eqref{eq:solvable-vanishing-formula-3}, we get for any $i\in \bN$ 
\begin{equation}\label{eq:solvable-vanishing-formula-4}
\coh_i \bigl(A,\Steinberg(\sG,\bF) \bigr)
\cong 
\bigoplus_{[u]\in A\backslash \sP(\bQ)/\sL(\bQ)} 
\coh_i \bigl(A_u,  u\Steinberg(\sL,\bF)\bigr).
\end{equation}

Let $\pi: \sP\to \sL$ be the projection. Consider 
\[
\widetilde A_u:=u^{-1}A_u u<\sL(\bQ).
\]
We claim that
\begin{equation}\label{eq:solvable-vanishing-formula-5}
\widetilde A_u<\pi(A).
\end{equation}
Indeed, if $a\in A_u$, then $a=u\ell u^{-1}$ for some $\ell\in\sL(\mathbb Q)$, and \(\pi(a)=\ell=u^{-1}au\).

The group $\widetilde A_u$ is torsion-free and solvable. It is finitely
generated because solvable subgroups of arithmetic groups are virtually
polycyclic; see \cite{Raghunathan}. It is also contained in an arithmetic subgroup of $\sL$, since $\pi$ maps arithmetic subgroups of $\sP$ into arithmetic subgroups of $\sL$ by \cite[Prop.~4.1]{PlatonovRapinchuk}. 

Since $\sP$ is a proper parabolic subgroup, the semisimple $\bQ$-rank of $\sL$ is strictly smaller than that of $\sG$. Thus the induction hypothesis applies to $\widetilde A_u<\sL(\bQ)$. We note that
\[
r_\sL=r_\sG,
\qquad
c_\sL+q_\sL=c_\sG+q_\sG, 
\qquad \delta(\sL)=\delta(\sG).
\]
Therefore, for every $i>\delta(\sG)$, for every $[u]\in A\backslash \sP(\bQ)/\sL(\bQ)$,
\begin{equation}\label{eq:solvable-vanishing-formula-6}
\coh_i \bigl(\widetilde A_u,  \Steinberg(\sL,\bF)\bigr)=0.
\end{equation}
Conjugation by $u$ identifies $\coh_i \bigl(\widetilde A_u,  \Steinberg(\sL,\bF)\bigr)$ with $\coh_i \bigl(A_u,  u\Steinberg(\sL,\bF)\bigr)$. Inserting \eqref{eq:solvable-vanishing-formula-6} in \eqref{eq:solvable-vanishing-formula-4}, we obtain the desired vanishing. The proof is finished.

\end{proof}

\subsection{Poincar\'e--Bieri--Eckmann--Borel--Serre duality}\label{subsec:proof-cohomology-vanishing}

Let $G_1,\cdots,G_n$ be non-compact connected center-free simple linear algebraic groups. Let $G=G_1\times \cdots \times G_n$ and let $\Gamma < G$ be a torsion-free irreducible lattice. Assume that the real rank of $G$ is $\ge 2$. Let $Y$ be the symmetric space, the quotient of $G$ by its maximal compact subgroup. Let $X=Y/\Gamma$ be the quotient manifold. Denote by $d$ the dimension of $X,Y$. Denote by $r$ the real rank of $G$, and $r_1,\cdots, r_n$ the ranks of $G_1,\cdots,G_n$, so $r=\sum r_i$. 
Since each $Y_i$ has dimension $\ge 2r_i$, we have $d\ge 2r$, so 
\begin{equation}\label{eq:dr}
d-1>r.
\end{equation}
As $\Gamma$ is torsion-free, we will identify it with $\pi_1(X)$.  

\subsubsection{Uniform lattice}\label{subsubsec:uniform-lattice}

Assume first that $\Gamma$ is a uniform lattice, so $X$ is an aspherical compact orientable manifold of dimension $d$. Then $\Gamma=\pi_1(X)$ satisfies Poincar\'e duality: for any $\bZ \Gamma$-module $M$, for any $i\in \{0,\cdots, d\}$, one has
\begin{equation}\label{eq:poincare-duality}
\coh^i(\Gamma, M)=\coh_{d-i}(\Gamma,M).
\end{equation}
Theorem~\ref{mainthm:cohomology-vanishing} is particularly simple for uniform lattices. We explain the proof for illustration:
\begin{proof}[Proof of Theorem~\ref{mainthm:cohomology-vanishing} in the uniform case]
Since $\Gamma$ is a uniform lattice, it is a classical fact that $K$ is virtually a free abelian group of rank $\leq r$ (use for example Lie--Kolchin Theorem and Godement's cocompactness criterion). Hence it has virtual cohomological dimension $\le r$ and
\begin{equation}\label{eq:uniform-solvable-vanishing-formula-1}
\coh_j(K,\bF)=0, \qquad j>r.
\end{equation}
Combining \eqref{eq:dr}, \eqref{eq:uniform-solvable-vanishing-formula-1} and Shapiro's lemma (see \cite[III.6.2]{BrownBook}), we obtain
\begin{equation}\label{eq:uniform-solvable-vanishing-formula-2}
\coh_{d-1}(\Gamma,\bF[\Gamma/K])=\coh_{d-1}(K,\bF)=0.
\end{equation}
Then the conclusion follows from Poincar\'e duality.
\end{proof}

\subsubsection{General case}\label{subsubsec:non-uniform-lattice}

Now we drop the assumption that $\Gamma$ is uniform. By Margulis' arithmeticity theorem \cite{MargulisBook} it is an arithmetic group. Let $\sH$ be a $\bQ$-algebraic group such that $\sH(\bR)/\sK=G$, where $\sK$ is a real compact factor, and such that $\Gamma$ is the image of an arithmetic group $\Gamma'\subset \sH(\bQ)$. We assume that $\Gamma'\to \Gamma$ is a bijection (this can be achieved by passing to a finite-index subgroup) and denote abusively both $\Gamma,\Gamma'$ by $\Gamma$. Denote by $q$ the $\bQ$-rank of $\sH$.  

\begin{rem}\label{rem:Q-rank-one}
If $\Gamma$ is non-uniform and if all factors $G_i$ have real rank one, then $q=1$, see \cite[Lem.~1.1]{PrasadRankone}.
\end{rem}

Borel--Serre~\cite{BorelSerre} showed that $\Gamma$ has cohomological dimension $d-q$ and satisfies the Bieri--Eckmann duality \cite{BieriEckmann}: for any field $\bF$, for any $\bF \Gamma$-module $M$, for any $i\in \{0,\cdots, d-q\}$, one has
\begin{equation}\label{eq:BE-duality}
\coh^i(\Gamma, M)=\coh_{d-q-i}(\Gamma,\Steinberg(\sH,\bF) \otimes M).
\end{equation}

\begin{proof}[Proof of Theorem~\ref{mainthm:cohomology-vanishing}]
We have already replaced $\Gamma$ with a finite-index torsion-free subgroup in the beginning paragraph of \S\ref{subsubsec:non-uniform-lattice}. In particular $K$ is also torsion-free. 

By duality \eqref{eq:BE-duality}, to prove Theorem~\ref{mainthm:cohomology-vanishing} it suffices to prove 
\[
\coh_{d-q-1}(\Gamma,\Steinberg(\sH,\bF) \otimes \bF[\Gamma/K])=0.
\]
Then by Shapiro's lemma (see \cite[III.6.2]{BrownBook}) it suffices to prove 
\begin{equation}\label{eq:solvable-vanishing-aim}
\coh_{d-q-1}\bigl(K,\Steinberg(\sH,\bF)\bigr)=0
\end{equation}
where $\Steinberg(\sH,\bF)$ is viewed as a $\bF K$-module. As $\sH$ is semisimple, we have $\delta(\sH)=r-q$, so $d-q-1>\delta(\sH)$ by \eqref{eq:dr}. Solvable subgroups of arithmetic groups are virtually polycyclic; see \cite{Raghunathan}. In particular $K$ is finitely generated. Then \eqref{eq:solvable-vanishing-aim} follows from Proposition~\ref{prop:solvable-steinberg-vanishing-reductive}.
\end{proof}

\section{Actions on CAT(0) cube complexes}\label{sec:cube-complex}
Once Theorem~\ref{mainthm:cohomology-vanishing} is established, our proof of Theorem~\ref{mainthm:cube-complex} is inspired by an argument of Cornulier, see \cite[Lem.~7.I.2]{CornulierFW}.

\begin{proof}[Proof of Theorem~\ref{mainthm:cube-complex}]
Up to replacing $\Gamma$ with a finite-index subgroup (see the end of the proof for why this is harmless), we may assume that $\Gamma$ is torsion-free and satisfies the conditions needed in the proofs of \S\ref{sec:duality}. 

Let $\rho:\Gamma\to \Isom (\sC)$ be an action on a cube complex $\sC$. Let $\sW$ be the set of all hyperplanes of $\sC$. Let $\bF_2$ be a field with two elements and let $\bF_2[\sW]$ be the abelian group whose elements are finite formal sums of hyperplanes. Fix a vertex $o\in \sC$. Define, for $\gamma\in \Gamma$,
\begin{equation}\label{eq:cocycle-wall-definition}
c(\gamma):=\sum_{W\in \operatorname{Sep}(o,\gamma o)} [W]\in \bF_2[\sW]
\end{equation}
where $\operatorname{Sep}(o,\gamma o)$ is the subset of $\sW$ consisting of hyperplanes separating $o$ from $\gamma o$. Note that for any three distinct vertices $x,y,z$, we have 
\begin{equation}\label{eq:symmetric-difference-separation}
\operatorname{Sep}(x,z)=\operatorname{Sep}(x,y)\triangle\operatorname{Sep}(y,z)
\end{equation}
where $\triangle$ stands for symmetric difference. Therefore for $g,h \in \Gamma$, we get
\begin{equation}\label{eq:symmetric-difference-separation-2}
\operatorname{Sep}(o,gho)=\operatorname{Sep}(o,go)\triangle\operatorname{Sep}(go,gho)=\operatorname{Sep}(o,go)\triangle g\operatorname{Sep}(o,ho).
\end{equation}
Since we are using $\bF_2$ coefficients, \eqref{eq:cocycle-wall-definition} and \eqref{eq:symmetric-difference-separation-2} imply that $c$ is a cocycle:
\begin{equation}\label{eq:c-is-a-cocycle}
c(gh)=c(g)+gc(h), \qquad c\in \operatorname{Z}^1\big(\Gamma, \bF_2[\sW]\big).
\end{equation}
Any arithmetic group is finitely generated, see \cite{BorelHarishChandra}. We fix a finite symmetric generating set $M$ of $\Gamma$. Define
\[
E:=\bigcup_{m\in M} \operatorname{Sep}(o,mo),
\]
which is a finite set of hyperplanes. Let $\aO^1,\cdots, \aO^s\subset \sW$ be the finitely many $\Gamma$-orbits of hyperplanes with non-empty intersection with $E$. Define
\[
V:=\bigoplus_{i=1}^s \bF_2[\aO^i] \subset \bF_2[\sW],
\]
which is a $\Gamma$-invariant linear subspace. By construction, we have $c(m)\in V$ for every $m\in M$. Then \eqref{eq:c-is-a-cocycle} implies, by induction on word length, that $c(\gamma)\in V$ for every $\gamma\in \Gamma$. Thus, in fact, $c\in \operatorname{Z}^1(\Gamma, V)$.

For each orbit $\aO^i$, choose a base point $W_i\in \aO^i$. Then, as $\bF_2 \Gamma$-modules,
\begin{equation}\label{eq:wall-orbit-module}
\bF_2[\aO^i]\cong \bF_2[\Gamma/\Gamma_i]
\end{equation}
where $\Gamma_i$ is the stabilizer of $W_i$. Since by our assumption each $\Gamma_i$ is solvable, Theorem~\ref{mainthm:cohomology-vanishing} implies that, for every $i$, 
\begin{equation}\label{eq:cube-wall-cohomology-vanishing}
\coh^1\big(\Gamma,\bF_2[\aO^i]\big)=0, \quad \text{thus } \, 
\coh^1\big(\Gamma,V\big)=0. 
\end{equation}
Therefore $c$ is a coboundary and there exists $a\in V$ such that
\begin{equation}\label{eq:c-is-a-coboundary}
c(\gamma)=\gamma a-a=\gamma a +a, \quad \text{for every }\gamma\in \Gamma.
\end{equation}

Let $A:=\Supp(a)$ be the support of $a$, i.e.\ the finite set of hyperplanes that appear in the expression of $a$ as a formal sum. Then by \eqref{eq:c-is-a-coboundary}, for every $\gamma\in \Gamma$, 
\begin{equation}\label{eq:support-control}
\Supp (c(\gamma))\subset \gamma A\cup A,\qquad  \#\Supp(c(\gamma)) \le 2\# A.
\end{equation}
Thus the combinatorial distance of $\sC$ satisfies, for every $\gamma\in \Gamma$,
\begin{equation}\label{eq:combinatorial-distance-bounded}
d(o,\gamma o)=\# \operatorname{Sep}(o,\gamma o)=\# \Supp(c(\gamma))\le 2\# A.
\end{equation}
In particular, the orbit $\Gamma\cdot o\subset \sC$ is bounded in combinatorial metric. Hence, by Gerasimov's theorem~\cite{Gerasimov} (see also \cite[7.G]{CornulierFW}, \cite[Prop.~2.3]{LonjouUrech}), there exists a $\Gamma$-fixed vertex in $\sC$. Gerasimov's theorem also justifies why at the beginning of the proof it was harmless to take a finite-index subgroup: boundedness would not change. 
\end{proof}

\section{Actions on infinite-dimensional hyperbolic spaces}
\label{sec:rigidity}

\subsection{Superrigidity}
Let $\mathcal V$ be a real Hilbert space and \(\aE=\bR\oplus \mathcal V\) be a real Hilbert space equipped with a Lorentzian form $B$ of signature $(1,\infty)$:
\[
B\bigl((s,u),(t,v)\bigr)=st-\langle u,v\rangle_{\mathcal V}.
\]
Let $\bH^\aE$ be a real hyperbolic space of infinite dimension, realized by the hyperboloid model (see for example \cite{CantatTits,BurgerIozziMonod}):
\[
\bH^\aE
=
\{(t,v)\in \aE :B\bigl((t,v),(t,v)\bigr)=1,\ t>0\}.
\]
For \(x,y\in \bH^\aE\), the hyperbolic metric on $\bH^\aE$ satisfies
\[
B(x,y)=\cosh d(x,y).
\]
The metric space $\bH^\aE$ is $\operatorname{CAT}(-1)$, thus of telescopic dimension $1$ in the sense of Caprace--Lytchak~\cite{CapraceLytchak}. 

A group action by isometries on $\bH^\aE$ is called non-elementary if it does not fix any point in the Gromov compactification $\overline{\bH^\aE}$ and does not preserve any geodesic in $\bH^\aE$. 

\begin{rem}
If $G$ is of Hermitian type and if $\Gamma$ is uniform and torsion-free, then $\Gamma$ is a Kähler group and Theorem~\ref{mainthm:hyperbolic-lattice-rigidity} in this case is due to Delzant--Py~\cite{DelzantPy12} whose proof does not rely on Monod's superrigidity.  
\end{rem}

\begin{proof}[Proof of Theorem~\ref{mainthm:hyperbolic-lattice-rigidity}]
Existence and uniqueness of $\bH^\rho$ is exactly \cite[Proposition~4.3]{BurgerIozziMonod} (cf. \cite{DelzantPy12}). By abuse of notation we still denote the restriction of our action to $\bH^\rho$ by $\rho$. 

Since $\rho$ is non-elementary, \cite[Proposition~1.8]{CapraceLytchak} (see also \cite[Appendix]{BaderDuchesneLecureux}) provides a non-empty closed convex $\rho$-invariant subset $C\subset \bH^\rho$ on which the $\Gamma$ action is minimal and reduced in the sense of \cite{MonodRigidity}. Then \cite[Theorem~6, p.3]{MonodRigidity} (see also \cite[Corollary~1.9]{CapraceLytchak} and \cite[Appendix]{BaderDuchesneLecureux}) asserts that $\rho$ extends to a continuous homomorphism $\rho_G^C:G\to \Isom(C)$. Note that our lattices satisfy the weakly compact and square-integrable conditions of \cite[Theorem~6]{MonodRigidity} by \cite[Appendix~B]{MonodRigidity}. 

Moreover, since $C$ is not isometric to a Hilbert space, the proof of \cite[Theorem~6]{MonodRigidity} in \cite[\S~6.4]{MonodRigidity} says that $\rho_G^C$ factors through a continuous homomorphism $\rho_i^C:G_i\to \Isom(C)$ for exactly one $i$. 

Now let us show that $\rho_i^C$ is induced by a continuous homomorphism $\rho_i:G_i\to \Isom(\bH^\rho)$. Note that $\bH^\rho$ is the hyperboloid of a closed linear subspace $\aE^\rho\subset \aE$ where the restriction of the Lorentzian form $B$ is still Lorentzian. By minimality of $\bH^\rho$, the closure of the linear subspace of $\aE^\rho$ generated by $C$ is $\aE^\rho$ itself. Let \(f\in\Isom(C)\). Define formally
\begin{equation}\label{eq:Monod-Convex-extend-to-Hilbert}
T_0^f\big(\sum_{j=1}^n \lambda_jx_j\big)
:=
\sum_{j=1}^n \lambda_jf(x_j),
\qquad x_j\in C.
\end{equation}
We need to prove that \(T_0^f\) is well-defined. Suppose that \(\sum\lambda_jx_j=0\), and consider \(v=\sum_{j=1}^n\lambda_jf(x_j)\). 
For every \(y\in C\), we have
\begin{align*}
B(v,f(y))
&=
\sum_{j=1}^n\lambda_jB(f(x_j),f(y)) 
=
\sum_{j=1}^n\lambda_j\cosh d(f(x_j),f(y)) \\
&=
\sum_{j=1}^n\lambda_j\cosh d(x_j,y) 
=
B\left(\sum_{j=1}^n\lambda_jx_j,y\right)
=0.
\end{align*}
Thus \(v=0\), proving that \(T_0^f\) is well-defined. The same calculation gives \(B(T_0^f u,T_0^f v)=B(u,v)\). Applying the construction to \(f^{-1}\) shows that \(T_0^f\) is bijective. Hence $T_0^f$ extends uniquely to a continuous isometry $T^f:\aE^\rho\to \aE^\rho$. Moreover \eqref{eq:Monod-Convex-extend-to-Hilbert} shows that the construction is continuous in $f$ and gives a group homomorphism $f\mapsto T^f$. 
\end{proof}

\begin{rem}\label{rem:Monod-Py}
As we explained in the introduction, the interesting case for our later applications would be the case where all factors $G_i$ have real rank one and do not have property~(T). In this case, if there is a non-elementary action as in Theorem~\ref{mainthm:hyperbolic-lattice-rigidity}, then each factor $G_i$ is necessarily $\SO(1,m)$ for some $m$. The reason is as follows. First, the rank-one groups $\operatorname{Sp}(m,1), m\ge 2$ and $F_4^{(-20)}$ have property~(T) (see \cite[\S~3.3]{BHV}). Second, Stolowicz~\cite{Stolowicz} proved that $\operatorname{PU}(1,m),m\ge 2$ do not have continuous non-elementary actions on infinite-dimensional real hyperbolic spaces. Thus the only possibilities are $\SO(1,m)$. For $\SO(1,m)$, continuous non-elementary actions on infinite-dimensional real hyperbolic spaces are completely classified by Monod--Py~\cite{MonodPyhyperbolic}. 
\end{rem}

\subsection{Cremona group}
\label{subsec:cremona-hyperbolic}

Let $S$ be a complex projective surface. The group of birational transformations $\Bir(S)$ acts faithfully by isometries on an infinite-dimensional real hyperbolic space $\bH^\aE$ depending on $S$; we refer to \cite{CantatTits} for the construction. 

Choose an affine chart $\bC^2\subset \bP^2$ with coordinates $(x,y)$, and define the de Jonquières subgroup (often abbreviated as Jonquières subgroup) as
\begin{align}
 \Jonq :=\{ & f\in \Bir(\bP^2): f(x,y)=\big(p(x),\frac{a(x)y+b(x)}{c(x)y+d(x)}\big),\notag
\\
& p\in \PGL(2,\bC), \begin{pmatrix} a&b\\c&d\end{pmatrix}\in \PGL(2,\bC(x))\}.\label{eq:Jonq-definition}
\end{align}
It depends on the choice of the coordinates; different coordinates yield a conjugation by an element of $\PGL(3,\bC)$. It fits into a split short exact sequence
\begin{equation}\label{eq:Jonquires-exact-sequence}
1\to \PGL(2,\bC(x))\to \Jonq \to
\PGL(2,\bC)\to 1,
\end{equation}
and is isomorphic to the semidirect product $\PGL(2,\bC(x))\rtimes \PGL(2,\bC)$. The action of $\Jonq$ on $\bH^\aE$ fixes a point on $\partial \bH^\aE$, see \cite{CantatTits}. Combining Theorem~\ref{mainthm:hyperbolic-lattice-rigidity} with results of Monod--Py and Blanc--Cantat we get:
\begin{lem}\label{lem:rigidity-hyperbolic-cremona}
Let $G_1,\cdots,G_n$ be non-compact connected center-free simple linear algebraic groups. Let $G=G_1\times \cdots \times G_n$ and let $\Gamma < G$ be an irreducible lattice. Assume that the real rank of $G$ is $\ge 2$. Let $S$ be a complex projective surface. Let $\rho:\Gamma\to \Bir(S)$ be a homomorphism of groups with infinite image. Then either the image $\rho(\Gamma)$ is virtually isotopically projectively regularizable, or $S$ is rational, $\rho(\Gamma)$ is up to conjugation in $\Bir(\bP^2)$ contained in $\Jonq$, and the composition $\Gamma\to \Jonq\to \PGL(2,\bC)$ has finite kernel. 
\end{lem}
\begin{proof}
Assume first by way of contradiction that the action of $\rho(\Gamma)$ on $\bH^\aE$ via the Cremona action is non-elementary. Then we deduce from Theorem~\ref{mainthm:hyperbolic-lattice-rigidity} that $\rho$ factors through a continuous homomorphism $\rho_i:G_i\to \Isom(\bH^\aE)$. We remark that $G_i$ necessarily has real rank one. Indeed, if it had higher real rank then it would have property~(T) and could not have non-elementary action on $\bH^\aE$; see \cite[\S~4.2.2]{CantatTits}. Then by Remark~\ref{rem:Monod-Py} $G_i$ is $\SO(1,m)$ for some $m$. By Monod--Py's classification~\cite[Theorem~B]{MonodPyhyperbolic}, there exists $0<t\le 1$ such that every element $g\in G_i$ satisfies
\begin{equation}\label{eq:Monod-Py-t}
\ell(\rho_i(g),\bH^\aE)=t\ell(g,\bH^m)
\end{equation}
where the right-hand side is the translation length in $\bH^m$ while the left-hand side is the translation length in $\bH^\aE$ via the Cremona action. Since $\Gamma$ is an irreducible lattice, its projection in $G_i$ is dense in the Euclidean topology. Therefore \eqref{eq:Monod-Py-t} implies that there exists $\gamma\in \Gamma$ such that $\ell(\rho(\gamma),\bH^\aE)$ is arbitrarily close to $0$. This contradicts $\rho(\gamma)\in \Bir(S)$ by the gap property of \cite[Corollary~2.7(2)]{BlancCantat}.

Therefore $\rho(\Gamma)$ is elementary, i.e.\ it either fixes a point in $\overline{\bH^\aE}$ or preserves a geodesic in $\bH^\aE$. 
If $\rho(\Gamma)$ fixes a point in $\bH^\aE$, then \cite[Proposition~3.10]{CantatTits} says that it is virtually isotopically projectively regularizable. 

If $\rho(\Gamma)$ preserves a geodesic in $\bH^\aE$, then by \cite[Theorem~4, and the unnumbered statement before \S~6]{DelzantPy12} it is either virtually cyclic or contained in the so-called toric subgroup (a group with a semidirect product structure $(\bC^*)^2\rtimes \GL(2,\bZ)$). Both alternatives contradict Margulis' normal subgroup theorem~\cite{MargulisBook} for $\Gamma$.

If it fixes a point in $\partial\bH^\aE$, then \cite[Proposition~5.12]{CantatTits} says that, up to conjugation in $\Bir(S)$, it preserves a rational fibration or it is contained in the automorphism group of a Halphen rational surface. Note that the automorphism group of a Halphen surface is virtually free abelian by \cite{Gizatullin} (see also \cite[Proposition~5.2]{CantatTits}). The latter alternative is impossible in our case again because of Margulis' normal subgroup theorem. 

It remains to examine the case where $\rho(\Gamma)$ preserves a rational fibration. Then there is an induced action of $\Gamma$ on the base of this rational fibration. If this base is $\bP^1$, then the surface $S$ is rational and $\rho(\Gamma)$ is up to conjugation contained in $\Jonq$. If this base is some other Riemann surface $C$, then, because $C$ has finite or virtually abelian automorphism group, Margulis' normal subgroup theorem implies that the $\rho(\Gamma)$ action on this base is finite. In other words, after passing to a finite-index subgroup, the $\Gamma$-action preserves the fibration fiberwise and can be seen as an algebraic family of actions on $\bP^1$. Hence, up to replacing $\Gamma$ with a finite-index subgroup, $\rho$ induces a homomorphism $\Gamma\to \PGL(2,\bC(C))$. By Margulis' superrigidity in the non-archimedean case applied to valuations on $\bC(C)$, the image of $\Gamma$ is bounded. Hence up to conjugation it has a finite-index subgroup contained in $\PGL(2,\bC)$. In particular it is virtually isotopically projectively regularizable.

If the action on the base of the rational fibration is infinite, then by Margulis' normal subgroup theorem, the kernel of this base action is finite. 
\end{proof}

\begin{lem}\label{lem:limited-type-lie-factors}
If $\Gamma$ is as in Lemma~\ref{lem:rigidity-hyperbolic-cremona} and $\rho:\Gamma\to \Jonq$ is a homomorphism with infinite image, then every $G_i$ is either $\PSL(2,\bR)$ or $\PGL(2,\bC)$.
\end{lem}
\begin{proof}
Consider the composition $\rho_B:\Gamma\to \Jonq\to \PGL(2,\bC)$ by using \eqref{eq:Jonquires-exact-sequence}. Assume first that $\rho_B$ has infinite image. Then Margulis' superrigidity implies that $\rho_B$ is induced by a continuous homomorphism $G\to \PGL(2,\bC)$. Therefore at least one $G_i$ admits a continuous homomorphism $G_i\to \PGL(2,\bC)$ with positive dimensional image. This forces this $G_i$ to be $\PSL(2,\bR)$ or $\PGL(2,\bC)$. Since $G$ admits an irreducible lattice, it is isotypic (see \cite[Cor.IX.4.5]{MargulisBook}). Hence each $G_i$ is either $\PSL(2,\bR)$ or $\PGL(2,\bC)$. 

Assume now that $\rho_B$ has finite image. Then replacing $\Gamma$ with a finite-index subgroup, we may assume that $\rho$ is a homomorphism from $\Gamma$ to $\PGL(2,\bC(x))$. By Margulis' superrigidity again, up to conjugation it has a finite-index subgroup contained in $\PGL(2,\bC)$. Then repeating the above argument finishes the proof. 
\end{proof}

\begin{proof}[Proof of Theorem~\ref{mainthm:Zimmer-no-action}]
Assume that there is a homomorphism $\rho:\Gamma\to \Bir(S)$ with infinite image. By Lemma~\ref{lem:rigidity-hyperbolic-cremona} it is either virtually isotopically projectively regularizable or up to conjugation contained in $\Jonq$. In the latter case Lemma~\ref{lem:limited-type-lie-factors} closes the proof. 

Therefore up to replacing $\Gamma$ with a finite-index subgroup and up to conjugation, we may assume that $\rho(\Gamma)$ is contained in $\Aut^0(S)$. Without loss of generality we can also assume that $\rho(\Gamma)$ does not preserve any rational fibration because otherwise the proof of Lemma~\ref{lem:limited-type-lie-factors} already closes the proof. We refer to \cite{CantatMilnor} for the information on $\Aut^0$ that we will use below. If $S$ had non-negative Kodaira dimension, then $\Aut^0(S)$ would be trivial or a complex torus, contradicting Margulis' normal subgroup theorem. Hence $S$ has negative Kodaira dimension; it is rational or birational to a ruled surface over a non-rational curve. If $S$ is birational to a non-rational ruled surface or if $S$ is obtained by blow-ups from a Hirzebruch surface, then $\Aut^0(S)$ preserves the rational fibration (the ruling), which is out of our consideration by our hypothesis. The only possibility left is $S=\bP^2$ and $\Aut^0(S)=\PGL(3,\bC)$. We conclude as in the proof of Lemma~\ref{lem:limited-type-lie-factors} by using Margulis' superrigidity with target group $\PGL(3,\bC)$.
\end{proof}

\section{Actions on the Jonquières complex}
\label{sec:Jonquieres}

\subsection{The Jonquières complex}\label{subsec:jonquires}
Lonjou--Urech constructed in \cite{LonjouUrech} an infinite-dimensional CAT(0) cube complex on which $\Bir(\bP^2)$ acts by isometries. Then Lonjou--Przytycki--Urech studied in \cite{LPU} a sub-complex $\fJ$ on which the de Jonquières subgroup acts by isometries, called the Jonquières complex. We summarize some features of the Jonquières complex needed in our proof and refer the detailed construction to \cite{LonjouUrech,LPU}. 

One first constructs, for $p\in \bP^1$, a pointed combinatorial tree $(X_p,x_p)$ encoding the relations of elementary transformations of a ruled surface along the fiber $\{p\}\times \bP^1$. Roughly speaking, the vertices of $X_p$ are conic bundles isomorphic to $\bP^1\times \bP^1$ outside the fiber $\{p\}\times \bP^1$, the edges of $X_p$ are single blow-ups or blow-downs between such conic bundles, and the choice of $x_p\in X_p$ is arbitrary. One can also imagine $X_p$ as a barycentric subdivision of the Bruhat--Tits tree of $\PGL(2,\bC(\!(X)\!))$, or as a subtree of the Berkovich projective line over $\bC(\!(X)\!)$. 

The Jonquières complex $\fJ$ is an infinite-dimensional CAT(0) cube complex isomorphic to the restricted product
\[
\fJ=\bigoplus_{p\in \bP^1} (X_p,x_p),
\]
whose vertices are sections $(y_k)_{k\in \bP^1}$ with $y_k\in X_k$ such that all but finitely many $y_k$ are equal to $x_k$, and whose cubes have the form $\prod_{k\in \bP^1}I_k$ where all but finitely many $I_k$ are equal to $\{x_k\}$ while the remaining $I_k$ are edges of $X_k$. The Jonquières group $\Jonq$ (see \eqref{eq:Jonq-definition}) acts on $\fJ$ by isometries of cube complexes. See \cite[Lemma~3.4]{LPU}.

\subsection{Proof of Theorem~\ref{mainthm:Lattice-regularizable}}
\begin{proof}
Without loss of generality, up to taking a finite-index subgroup, we assume that $\Gamma$ is torsion-free.

By Lemma~\ref{lem:rigidity-hyperbolic-cremona}, we can assume without loss of generality that $S=\bP^2$ and $\rho:\Gamma\to \Bir(\bP^2)$ has image contained in $\Jonq$. Hence, there is an induced action $\xi:\Gamma\to \Isom(\fJ)$ on the Jonquières complex. 

We remark first that every edge of $\fJ$ is an edge of some tree $X_p$. Then we remark that, for distinct points $p,q\in \bP^1$, an edge of $X_p$ is never a parallel edge of any edge of $X_q$ in any cube of $\fJ$ because of the restricted product structure of $\fJ$. Therefore the hyperplanes of \(\fJ\) are naturally indexed by pairs \((p,e)\), where \(p\in \bP^1\) and \(e\) is an edge of the tree \(X_p\). In particular, if an element stabilizes such a hyperplane, its image in \(\PGL(2,\bC)\) stabilizes the point \(p\). More precisely, if $G\subset \Gamma$ is a subgroup such that $\xi (G)$ fixes a hyperplane corresponding to an edge in $X_p$ for some $p\in \bP^1$, then the image of the composition $G\to \Jonq\to \PGL(2,\bC)$ fixes $p$. As $\Gamma$ is torsion-free, the last assertion of Lemma~\ref{lem:rigidity-hyperbolic-cremona} says that the composition $\Gamma\to \Jonq\to \PGL(2,\bC)$ is injective. Hence every stabilizer of a hyperplane injects into the stabilizer in $\PGL(2,\bC)$ of some $p\in \bP^1$. In particular, every hyperplane stabilizer is solvable. By Theorem~\ref{mainthm:cube-complex} there is a $\Gamma$-fixed vertex. This implies, by construction of $\fJ$, that $\rho(\Gamma)$ is projectively regularizable, see for example \cite[Prop.~3.11]{LonjouUrech}. A projectively regularizable group which is not virtually isotopically projectively regularizable must contain a loxodromic element or a Halphen twist, while the Jonquières group does not contain such elements, see \cite[\S~6.5]{CantatTits}. The proof is finished.
\end{proof}

\section{Actions on Hirzebruch surfaces}\label{sec:hirzebruch}
Consider irreducible lattices in Lie groups of the form 
\[
G=\PSL(2,\bR)^n\times \PGL(2,\bC)^m,\qquad m+n\ge 2.
\]
By Lemma~\ref{lem:limited-type-lie-factors} these are the only higher-rank lattices that have infinite action on surfaces by birational transformations preserving a rational fibration. 
Let $\Gamma$ be an irreducible lattice in $G$. 

For $k>1$, the automorphism group of the Hirzebruch surface $\sF_k$ is connected and is a semidirect product (see \cite{Maruyama})
\begin{equation}\label{eq:hirzebruch-automorphism}
\Aut(\sF_k) =\coh^0(\bP^1,\sO(k))\rtimes 
\bigl(\GL(2,\bC)/\mu_k\bigr)
\end{equation}
where $\mu_k=\{\zeta I_2\in \GL(2,\bC): \zeta^k=1\}$. 
Moreover $\Aut(\sF_k)$ is linear with unipotent radical $\coh^0(\bP^1,\sO(k))$. 

Let $\rho:\Gamma\to \Aut(\sF_k)$ be a homomorphism with infinite image. Then by Margulis' superrigidity it is induced by a continuous homomorphism from $G$ to $\Aut(\sF_k)$, which factors through one of the factors. On the level of Lie algebras, there is a morphism
\[
\mathfrak{g}\to \coh^0(\bP^1,\sO(k))\rtimes \mathfrak{Lie}(\GL(2,\bC)/\mu_k), \quad x\mapsto (b(x),a(x)).
\]
Here $b$ defines a $1$-cocycle in $\coh^1(\mathfrak{g},\coh^0(\bP^1,\sO(k)))$ via $x\mapsto a(x)$. By Whitehead's lemma~\cite[Cor.~7.8.10]{Weibel} $b$ is a coboundary. In other words, up to conjugation by an element of $\coh^0(\bP^1,\sO(k))$, the image of $\rho$ is necessarily contained in the subgroup $\big(\GL(2,\bC)/\mu_k\big)<\Aut(\sF_k)$, and is thus induced by one of the natural maps from $\Gamma$ to $\GL(2,\bC)$.

\addtocontents{toc}{\protect\setcounter{tocdepth}{-1}}
\section*{Acknowledgements}
\addtocontents{toc}{\protect\setcounter{tocdepth}{2}}
I would like to thank Serge Cantat for introducing me to the topic and Christian Urech for related discussions. 

I acknowledge partial support from the French National Research Agency under the projects GAG (ANR-24-CE40-3526-01) and DynAtrois (ANR-24-CE40-1163).

\bibliographystyle{amsplain}
\bibliography{biblio}

\providecommand{\bysame}{\leavevmode\hbox to3em{\hrulefill}\thinspace}
\providecommand{\MR}{\relax\ifhmode\unskip\space\fi MR }
\providecommand{\MRhref}[2]{%
  \href{http://www.ams.org/mathscinet-getitem?mr=#1}{#2}
}
\providecommand{\href}[2]{#2}
\begin{thebibliography}{10}

\bibitem{APS}
Avner Ash, Andrew Putman, and Steven~V Sam, \emph{Homological vanishing for the
  {Steinberg} representation}, Compos. Math. \textbf{154} (2018), no.~6,
  1111--1130.

\bibitem{BaderDuchesneLecureux}
Uri Bader, Bruno Duchesne, and Jean L{\'e}cureux, \emph{Furstenberg maps for
  {CAT}(0) targets of finite telescopic dimension}, Ergodic Theory Dyn. Syst.
  \textbf{36} (2016), no.~6, 1723--1742.

\bibitem{BeauvilleBook}
Arnaud Beauville, \emph{Complex algebraic surfaces.}, 2nd ed. ed., Lond. Math.
  Soc. Stud. Texts, vol.~34, Cambridge: Cambridge Univ. Press, 1996.

\bibitem{BHV}
Bachir Bekka, Pierre de~la Harpe, and Alain Valette, \emph{Kazhdan's property},
  New Math. Monogr., vol.~11, Cambridge: Cambridge University Press, 2008.

\bibitem{BieriEckmann}
Robert Bieri and Beno Eckmann, \emph{Groups with homological duality
  generalizing {Poincar{\'e}} duality}, Invent. Math. \textbf{20} (1973),
  103--124.

\bibitem{BlancAlgGroup}
J{\'e}r{\'e}my Blanc, \emph{Algebraic subgroups of the {Cremona} group},
  Transform. Groups \textbf{14} (2009), no.~2, 249--285.

\bibitem{BlancCantat}
J{\'e}r{\'e}my Blanc and Serge Cantat, \emph{Dynamical degrees of birational
  transformations of projective surfaces}, J. Am. Math. Soc. \textbf{29}
  (2016), no.~2, 415--471.

\bibitem{BorelHarishChandra}
Armand Borel and Harish-Chandra, \emph{Arithmetic subgroups of algebraic
  groups}, Ann. Math. (2) \textbf{75} (1962), 485--535.

\bibitem{BorelSerre}
Armand Borel and Jean-Pierre Serre, \emph{Corners and arithmetic groups
  ({Appendice}: {Arrondissement} des varietes a coins par {A}. {Douady} et {L}.
  {Herault})}, Comment. Math. Helv. \textbf{48} (1973), 436--491.

\bibitem{BorelTits}
Armand Borel and Jacques Tits, \emph{{\'E}l{\'e}ments unipotents et
  sous-groupes paraboliques de groupes r{\'e}ductifs. {I}}, Invent. Math.
  \textbf{12} (1971), 95--104.

\bibitem{BFH2}
Aaron Brown, David Fisher, and Sebastian Hurtado, \emph{Zimmer's conjecture for
  actions of {{\(\mathrm{SL}(m,\mathbb{Z})\)}}}, Invent. Math. \textbf{221}
  (2020), no.~3, 1001--1060.

\bibitem{BFH1}
\bysame, \emph{Zimmer's conjecture: subexponential growth, measure rigidity,
  and strong property ({T})}, Ann. Math. (2) \textbf{196} (2022), no.~3,
  891--940.

\bibitem{BrownBook}
Kenneth~S. Brown, \emph{Cohomology of groups}, Grad. Texts Math., vol.~87,
  Springer, 1982.

\bibitem{BrownBuildings}
\bysame, \emph{Buildings}, New York etc.: Springer-Verlag, 1989.

\bibitem{BurgerIozziMonod}
Marc Burger, Alessandra Iozzi, and Nicolas Monod, \emph{Equivariant embeddings
  of trees into hyperbolic spaces}, Int. Math. Res. Not. \textbf{2005} (2005),
  no.~22, 1331--1369.

\bibitem{CantatZimmer}
Serge Cantat, \emph{A {K{\"a}hlerian} version of a conjecture of {Robert} {J}.
  {Zimmer}}, Ann. Sci. {\'E}c. Norm. Sup{\'e}r. (4) \textbf{37} (2004), no.~5,
  759--768.

\bibitem{CantatTits}
\bysame, \emph{On the groups of birational transformations of surfaces}, Ann.
  Math. (2) \textbf{174} (2011), no.~1, 299--340.

\bibitem{CantatMilnor}
\bysame, \emph{Dynamics of automorphisms of compact complex surfaces},
  Frontiers in complex dynamics. In celebration of John Milnor's 80th
  birthday., Princeton, NJ: Princeton University Press, 2014, pp.~463--514.

\bibitem{CantatCornulier}
Serge Cantat and Yves de~Cornulier, \emph{Commensurating actions of birational
  groups and groups of pseudo-automorphisms}, J. {\'E}c. Polytech., Math.
  \textbf{6} (2019), 767--809.

\bibitem{CantatXie}
Serge Cantat and Junyi Xie, \emph{Algebraic actions of discrete groups: the
  {{\(p\)}}-adic method}, Acta Math. \textbf{220} (2018), no.~2, 239--295.

\bibitem{CantatZeghib}
Serge Cantat and Abdelghani Zeghib, \emph{Holomorphic actions, {Kummer}
  examples, and {Zimmer} program}, Ann. Sci. {\'E}c. Norm. Sup{\'e}r. (4)
  \textbf{45} (2012), no.~3, 447--489.

\bibitem{CapraceLytchak}
Pierre-Emmanuel Caprace and Alexander Lytchak, \emph{At infinity of
  finite-dimensional {CAT}{{\((0)\)}} spaces}, Math. Ann. \textbf{346} (2010),
  no.~1, 1--21.

\bibitem{CornulierLattices}
Yves Cornulier, \emph{Irreducible lattices, invariant means, and commensurating
  actions}, Mathematische Zeitschrift \textbf{279} (2015), no.~1--2, 1--26.

\bibitem{CornulierFW}
\bysame, \emph{Group actions with commensurated subsets, wallings and cubings},
  2016, arXiv:1302.5982v2.

\bibitem{CornulierCremona}
\bysame, \emph{Regularization of birational actions of {FW} groups},
  Confluentes Math. \textbf{12} (2020), no.~2, 3--10.

\bibitem{DSVX}
Danijela Damjanovic, Ralf Spatzier, Kurt Vinhage, and Disheng Xu, \emph{The
  {Zimmer} {Program} for partially hyperbolic actions}, Acta Math. (to appear).

\bibitem{DelzantPy12}
Thomas Delzant and Pierre Py, \emph{K{\"a}hler groups, real hyperbolic spaces
  and the {Cremona} group. {With} an appendix by {Serge} {Cantat}}, Compos.
  Math. \textbf{148} (2012), no.~1, 153--184.

\bibitem{Deserti}
Julie D{\'e}serti, \emph{Cremona group and complex dynamics: an approach to the
  {Zimmer} conjecture}, Int. Math. Res. Not. \textbf{2006} (2006), no.~11, 27.

\bibitem{Gerasimov}
V.~N. Gerasimov, \emph{Fixed-point-free actions on cubings}, Sib. Adv. Math.
  \textbf{8} (1997), no.~1, 91--109.

\bibitem{GhysCircle}
{\'E}tienne Ghys, \emph{Lattice actions on the circle}, Invent. Math.
  \textbf{137} (1999), no.~1, 199--231.

\bibitem{Gizatullin}
M.~Kh. Gizatullin, \emph{Rational {G}-surfaces}, Izv. Akad. Nauk SSSR, Ser.
  Mat. \textbf{44} (1980), 110--144.

\bibitem{LPU}
Anne Lonjou, Piotr Przytycki, and Christian Urech, \emph{Finitely generated
  subgroups of algebraic elements of plane {Cremona} groups are bounded}, J.
  {\'E}c. Polytech., Math. \textbf{11} (2024), 1011--1028.

\bibitem{LonjouUrech}
Anne Lonjou and Christian Urech, \emph{Actions of {Cremona} groups on
  {CAT}{{\((0)\)}} cube complexes}, Duke Math. J. \textbf{170} (2021), no.~17,
  3703--3743.

\bibitem{MargulisBook}
G.~A. Margulis, \emph{Discrete subgroups of semisimple {Lie} groups}, Ergeb.
  Math. Grenzgeb., 3. Folge, vol.~17, Berlin etc.: Springer-Verlag, 1991.

\bibitem{Maruyama}
Masaki Maruyama, \emph{On automorphism groups of ruled surfaces}, J. Math.
  Kyoto Univ. \textbf{11} (1971), 89--112.

\bibitem{MonodRigidity}
Nicolas Monod, \emph{Superrigidity for irreducible lattices and geometric
  splitting}, J. Am. Math. Soc. \textbf{19} (2006), no.~4, 781--814.

\bibitem{MonodPyhyperbolic}
Nicolas Monod and Pierre Py, \emph{An exotic deformation of the hyperbolic
  space}, Am. J. Math. \textbf{136} (2014), no.~5, 1249--1299.

\bibitem{PlatonovRapinchuk}
Vladimir Platonov and Andrei Rapinchuk, \emph{Algebraic groups and number
  theory. {Transl}. from the {Russian} by {Rachel} {Rowen}}, Pure Appl. Math.,
  Academic Press, vol. 139, Boston, MA: Academic Press, 1994.

\bibitem{PrasadRankone}
Gopal Prasad, \emph{Strong rigidity of {Q}-rank 1 lattices}, Invent. Math.
  \textbf{21} (1973), 255--286.

\bibitem{Raghunathan}
M.~S. Raghunathan, \emph{Discrete subgroups of {Lie} groups}, Ergeb. Math.
  Grenzgeb., vol.~68, Springer-Verlag, Berlin, 1972.

\bibitem{Reeder}
Mark Reeder, \emph{The {Steinberg} module and the cohomology of arithmetic
  groups}, J. Algebra \textbf{141} (1991), no.~2, 287--315.

\bibitem{Stolowicz}
Gonzalo Emiliano~Ruiz Stolowicz, \emph{{Real hyperbolic representations of
  $\text{PU}(1,n)$}}, Bulletin of the Belgian Mathematical Society - Simon
  Stevin \textbf{33} (2026), no.~2, 145 -- 165.

\bibitem{Weibel}
Charles~A. Weibel, \emph{An introduction to homological algebra}, Camb. Stud.
  Adv. Math., vol.~38, Cambridge: Cambridge University Press, 1994.

\bibitem{ZimmerProgramICM}
Robert~J. Zimmer, \emph{Actions of semisimple groups and discrete subgroups},
  Proc. {Int}. {Congr}. {Math}., {Berkeley}/{Calif}. 1986, {Vol}. 2, 1247-1258,
  1987.

\end{thebibliography}
\end{document}